\documentclass[11pt]{article}
\usepackage{bbm}
\usepackage[latin1]{inputenc}
\usepackage{amsmath,amsthm,amssymb}
\usepackage{amsfonts}
\usepackage{amsmath,amsthm,amssymb,amscd}
\usepackage{latexsym}
\usepackage{color}
\usepackage{graphicx}
\usepackage{mathrsfs}

\makeatletter \@addtoreset{equation}{section} \makeatother

\newtheorem{theorem}{Theorem}[section]

\newtheorem{proposition}{Proposition}[section]
\newtheorem{lemma}{Lemma}[section]
\newtheorem{remark}{Remark}[section]

\begin{document}
\title{\sc On degenerate fractional Schr\"{o}dinger--Kirchhoff-Poisson equations
with upper critical nonlinearity and electromagnetic fields }
\author{{\sc Zhongyi Zhang$^{a,}$\footnote{{{Corresponding author}}}
and  Du\v{s}an D. Repov\v{s}$^{b,c,d}$}\\
$^{\small\mbox{a}}${\small
Nanjing Vocational College of Information Technology,  Nanjing, Jiangsu, 210023, China, 
 {\it zhyzhang88@126.com}}\\
$^{\small\mbox{b}}${\small Faculty of Education, University of Ljubljana, Ljubljana, 1000,
Slovenia, {\it dusan.repovs@pef.uni-lj.si}}\\
$^{\small\mbox{c}}${\small Faculty of
Mathematics and Physics, University of Ljubljana, Ljubljana, 1000,
Slovenia, {\it dusan.repovs@fmf.uni-lj.si}}\\
$^{\small\mbox{d}}${\small  Institute of Mathematics, Physics and Mechanics,  Ljubljana, 1000,
Slovenia, {\it dusan.repovs@guest.arnes.si}}}
\date{}
\maketitle
%%%%%5%%%%%%%%%%%%%%%%%%%%%%%%%%%%%%%%%%%%%%%%%
\begin{abstract}
{This paper intends to study the following degenerate fractional
Schr\"{o}dinger--Kirchhoff-Poisson equations with critical
nonlinearity and electromagnetic fields in $\mathbb{R}^3$
\begin{equation*}
\left\{
\begin{array}{lll}
\varepsilon^{2s}M([u]_{s,A}^2)(-\Delta)_{A}^su + V(x)u + \phi u  =
k(x)|u|^{r-2}u +
\left(\mathcal{I}_\mu*|u|^{2_s^\sharp}\right)|u|^{2_s^\sharp-2}u, \ \ &x\in \mathbb{R}^3,\smallskip\smallskip\\
(-\Delta)^t\phi = u^2,  \ \  &x\in \mathbb{R}^3,
\end{array}\right.
\end{equation*}
where $\varepsilon > 0$ is a positive parameter, $3/4<s<1$, $0 < t <
1$,  $V$ is an electric potential satisfying
some suitable assumptions,
 $0 < k_\ast \leq k(x) \leq k^\ast$, $\mathcal{I}_\mu(x) =
|x|^{3-\mu}$ with $0<\mu<3$, $2_s^\sharp =\frac{3+\mu}{3-2s}$ and $2
< r < 2_s^\sharp$. With the help of the concentration compactness
principle and variational methods, 
 together with some careful
analytical 
methods,  we establish the existence and multiplicity of
solutions for the above problem when $\varepsilon \rightarrow 0$  in the
degenerate cases, that is when the Kirchhoff term $M$  vanishes at zero.
}\medskip

\medskip

\emph{\it Math. Subj. Classif. (2020):} 35A15, 35B99, 35J60, 47G20.
\end{abstract}

%%%%%%%%%%%%%%%%%%%%%%%%%%%%%%%%%%%%%%%%%%%%%
\section{Introduction}

In this paper, we intend to study the following degenerate
fractional Schr\"{o}dinger--Kirchhoff-Poisson equations with upper
critical nonlinearity and electromagnetic fields in $\mathbb{R}^3$
\begin{equation}\label{e1.1}
\left\{
\begin{array}{lll}
\varepsilon^{2s}M([u]_{s,A}^2)(-\Delta)_{A}^su + V(x)u + \phi u  =
k(x)|u|^{r-2}u +
\left(\mathcal{I}_\mu*|u|^{2_s^\sharp}\right)|u|^{2_s^\sharp-2}u, \ \ &x\in \mathbb{R}^3,\smallskip\smallskip\\
(-\Delta)^t\phi = u^2,  \ \  &x\in \mathbb{R}^3,
\end{array}\right.
\end{equation}
where $\varepsilon > 0$ is a positive parameter, $3/4<s<1$, $0 < t <
1$, $\mathcal{I}_\mu(x) = |x|^{3-\mu}$ with $0<\mu<3$,  $0 <
k_\ast \leq k(x) \leq k^\ast$,  $2_s^\sharp =\frac{3+\mu}{3-2s}$ is
the upper Sobolev critical exponent ,
 $2 < r < 2_s^\sharp$, $V$ is
an electric potential, $(-\Delta)_{A}^s$ and $A$ are called the
magnetic operator and magnetic potential, respectively. 

This problem
belongs to the 
fractional Schr\"{o}dinger-Poisson systems due to  the
potential $\phi$ satisfying a nonlinear fractional Poisson equation
in problem \eqref{e1.1}. According to d'Avenia and Squassina
\cite{da}, the fractional operator $(-\Delta)_{A}^s$ is defined by
\begin{eqnarray*}
(-\Delta)_{A}^s u(x) := 2\lim_{\varepsilon \rightarrow 0}
\int_{\mathbb{R}^N \setminus
B_\varepsilon(x)}\frac{u(x)-e^{i(x-y)\cdot
A(\frac{x+y}{2})}u(y)}{|x-y|^{N+2s}}dy, \quad x\in  \mathbb{R}^N,
\end{eqnarray*}
and magnetic potential $A$ is given by
\begin{eqnarray*}
[u]_{s,A}^2 := \iint_{\mathbb{R}^{2N}}\frac{|u(x)-e^{i(x-y)\cdot
A(\frac{x+y}{2})}u(y)|^2}{|x-y|^{N+2s}}dx dy.
\end{eqnarray*}
Throughout the paper,  the electric potential $V$ and Kirchhoff
function $M$  will satisfy the following assumptions:
\begin{itemize}
\item[($\mathcal {V}$)]  $V(x) \in C(\mathbb{R}^3, \mathbb{R})$, $V(0) = \min_{x\in \mathbb{R}^3} V(x) = 0$
and the set $V^{b} = \{x \in \mathbb{R}^3: V(x) < b\}$ has finite
Lebesgue measure, where $b > 0$ is a positive constant.

\item[$(\mathfrak{M})$] $(M_1)$ There exists $\sigma \in (1, 2_s^\sharp/2)$ satisfying
$\sigma\mathcal {M}(t)\geq M(t)t$ for all $t\geq0$, where $\mathcal
{M}(t)=\int_0^tM(s)ds$.

$(M_2)$ There exists $m_1 > 0$ such that $M(t) \geq m_1
t^{\sigma-1}$ for all $t \in \mathbb{R}^+$ and $M(0) = 0$.
\end{itemize}
\begin{remark}\label{rem1.1}
The typical function that satisfies conditions $(M_1)$-$(M_2)$ is
given by $M(t)=a+b\,t^{\sigma-1}$ for $t\in\mathbb{R}^+_0$, where
$a\in\mathbb{R}^+_0$, $b\in\mathbb{R}^+_0,$ and $a+b>0$. In
particular, when $M(t) \geq d > 0$ for some $d$ and all $t \geq 0$,
this case is said to be nondegenerate, while it is called
degenerate if $M(0) = 0$ and $M(t)> 0$ for $t
> 0$. However, in proving the compactness condition, 
the degenerate and the nondegenerate case are completely different
 and
it's more complicated in the degenerate case. In this paper, we deal
with the critical fractional Schr\"{o}dinger--Kirchhoff-Poisson
equations with electromagnetic fields in the degenerate case.
\end{remark}

Our motivation to study problem \eqref{e1.1} mainly comes from the
application of the fractional magnetic operator. We note that  the
equation with fractional magnetic operator often arises as a model
for various physical phenomena, in particular in the study of the
infinitesimal generators of L\'{e}vy stable diffusion processes
\cite{di}.  Also, several
papers on nonlocal operators and
on their applications exist, and hence we can refer
interested readers to \cite{du, la, to, x3, x1}. In order to further
research this type of equation by variational methods,  many
 have established the basic properties of fractional Sobolev
spaces, readers are referred to \cite{MRS, pa}.

First,  we make a quick overview of the literature on the magnetic
Schr\"{o}dinger equation without the Poisson term. For example,   there
are works on the magnetic Schr\"{o}dinger equation
\begin{equation}\label{e1.2}
-(\nabla u-{\rm i} A)^2u + V(x)u = f(x, |u|)u,
\end{equation}
where the  magnetic operator in Eq. \eqref{e1.2} is given by
\begin{displaymath}
-(\nabla u-{\rm i} A)^2u = -\Delta u +2iA(x)\cdot\nabla u +
|A(x)|^2u + iu \mbox{div} A(x).
\end{displaymath}
Squassina and Volzone \cite{sq1} proved that, up to
correcting the operator by the factor $(1-s)$, one has 
$(-\Delta)^s_A u \rightarrow -(\nabla u-{\rm i} A)^2u$ as
$s\rightarrow1$. Thus, up to normalization, the nonlocal case can be
seen as an approximation of the local one. Recently, many
researchers have paid attention to 
 equations with  fractional magnetic operator. In
particular,  Mingqi et al.  \cite{MPSZ}  obtained some existence
results of Schr$\ddot{\mbox{o}}$dinger--Kirchhoff type equation
involving the magnetic operator
\begin{equation}\label{e1.3}
M([u]_{s,A}^2)(-\Delta)_A^su+V(x)u=f(x,|u|)u\quad \text{in
$\mathbb{R}^N$},
\end{equation}
where $f$ satisfies the subcritical growth condition. For the
critical growth case, Zhang et al. \cite{bin}  first considered
the following fractional Schr$\ddot{\mbox{o}}$dinger equations:
\begin{eqnarray}\label{eq12}
\varepsilon^{2s}(-\Delta)^{s}_{A_{\varepsilon}}u+V(x)u=f(x,|u|)u+K(x)|u|^{2_{\alpha}^{*}-2}u\quad\quad
\mbox{in}\ \mathbb{R}^{N}.
\end{eqnarray}
They proved  the existence of  ground state solution $u_{\varepsilon}$ by using variational methods. For the non--degenerate case, Liang  et
al. \cite{liang3} proved the  existence and multiplicity of
solutions for a class of Schr$\ddot{\mbox{o}}$dinger--Kirchhoff type
equation. Ambrosio \cite{am7} obtained the existence and concentration results for some fractional Schr\"{o}dinger equations
in $\mathbb{R}^N$ with magnetic fields. As for other results, we refer to \cite{am8, am9, liang5, liang6} and
to the
references therein. We draw the attention of the reader to the degenerate case
involving the magnetic operator in Liang et al. \cite{liang4} and
Mingqi et al. \cite{MPSZ}.

On the other hand, for the case $A \equiv 0$ of problem
\eqref{e1.1}, some researchers  used various methods to study
this kind of problem. For example, using the perturbation approach,
Zhang et al.  \cite{zhj} obtained the existence results for the
fractional Schr\"{o}dinger-Poisson system with a general subcritical
or critical nonlinearity. In \cite{mu}, the authors looked for the
number of positive solutions for a class of doubly singularly
perturbed fractional Schr\"{o}dinger-Poisson system via the
Lyusternik-Schnirelmann category. Liu  \cite{liu} studied
the existence of multi-bump solutions for the fractional
Schr\"{o}dinger-Poisson system by means of the Lyapunov-Schmidt
reduction method. By using the non-Nehari manifold approach, Chen
and Tang  \cite{chen2} proved the existence of ground state
solutions for fractional Schr\"{o}dinger-Poisson system. For more
related results, see recent works \cite{am2, am5, am6,
fis2, gi, liang7, sh, song, te1, wang1, zhao} and the references
therein.

Once we turn our attention to the Schr\"{o}dinger--Kirchhoff-Poisson
equations with electromagnetic fields, we immediately see that the
literature is relatively scarce. In this case, we can cite the
recent works \cite{am3, am4, liu1}.  Ambrosio
\cite{am3} proved concentration results for a class of
fractional Schr\"{o}dinger-Poisson type equation with magnetic field
and subcritical growth. For the critical growth case, Ambrosio
\cite{am4} also obtained  the multiplicity and concentration of
nontrivial solutions to  the fractional Schr\"{o}dinger-Poisson
equation with magnetic field. However, to the best of our knowledge,
semiclassical solutions to the degenerate fractional
Schr\"{o}dinger--Kirchhoff-Poisson equations with critical
nonlinearity and electromagnetic fields \eqref{e1.1} have not
been considered until now.

Inspired by the previously mentioned works, our main objective is to
study the critical fractional Schr\"{o}dinger--Kirchhoff-Poisson
equations with electromagnetic fields in degenerate cases.  In our proofs
we use the concentration compactness
principle and variational methods. We also invoke some
minimax arguments.  Moreover, due to the appearance of the critical
term and degenerate nature of the Kirchhoff coefficient, the Sobolev
embedding does not possess the compactness, therefore we need
some technical estimates.

We are now in a position to state our main results - two existence theorems.
\begin{theorem}\label{the1.1} Assume that conditions $(\mathcal {V})$  and $(\mathfrak{M})$ hold.
Then for every $\kappa > 0$, there is $\mathcal {E}_\kappa > 0$ such
that if $0 < \varepsilon <  \mathcal {E}_\kappa$, then problem
\eqref{e1.1} has at least one solution $u_\varepsilon$. Moreover,
$u_\varepsilon \rightarrow 0$ in $E$ as $\varepsilon \rightarrow 0$.
\end{theorem} 
\begin{theorem}\label{the1.2} Assume that conditions $(\mathcal {V})$  and $(\mathfrak{M})$ hold. Then for every $m \in \mathbb{N}$ and $\kappa > 0$, there
is $\mathcal {E}_{m\kappa}
> 0$ such that if $0 < \varepsilon < \mathcal {E}_{m\kappa}$, then problem \eqref{e1.1} has at least $m$ pairs of
solutions $u_{\varepsilon,i}$, $u_{\varepsilon,-i}$,
$i=1,2,\cdots,m$. Moreover, $u_{\varepsilon,i}\rightarrow 0$ in $E$
as $\varepsilon \rightarrow 0$, $i=1,2,\cdots,m$.
\end{theorem}
The main feature of our paper is  establishing  results for
degenerate fractional Schr\"{o}dinger--Kirchhoff-Poisson equations
\eqref{e1.1} under the critical nonlinearity and electromagnetic
fields. The lack of compactness can lead to a lot of difficulties.
In order to overcome this challenge, we shall use the
concentration-compactness principles for fractional Sobolev spaces
due to \cite{li2, PP, zhang1}, and prove the $(PS)_c$ condition at
special levels $c$. On the other hand, we shall need to develop new
techniques to construct sufficiently small minimax levels.

The plan of this paper is the following:  In Section~\ref{sec2}, we
give some basic definitions of fractional Sobolev space and their
properties.  In Section~\ref{sec3}, we show some compactness lemmas
for the functional associated to our problem.  Section~\ref{sec4}
deals with the existence and multiplicity results for
problem~\eqref{e1.1}.

%%%%5%%%%%%%%%%%%%%%%%%%%%%%%%%%%%%%%%%%%%%%%%%%%%5%%%%%%%%%%%%%%%%%%%%%%%%%%%%%%%%%%%%%%%%%
\section{Preliminaries}\label{sec2}

In this section,  we have collected some known results for the readers
convenience and  for our later use.
For any $s \in (0, 1)$, fractional Sobolev space
$H_{A}^{s}(\mathbb{R}^3,\mathbb{C})$ is defined by
\begin{displaymath}
H_{A}^{s}(\mathbb{R}^3,\mathbb{C})  =  \left\{u \in
L^2(\mathbb{R}^N,\mathbb{C}): [u]_{s,A} < \infty\right\}
\end{displaymath}
 $[u]_{s,A}$ is the  Gagliardo semi-norm 
\begin{displaymath}
[u]_{s,A}  = \left(\iint_{\mathbb{R}^{6}}\frac{|u(x)-e^{i(x-y)\cdot
A(\frac{x+y}{2})}u(y)|^2}{|x-y|^{N+2s}}dx dy\right)^{1/2}
\end{displaymath}
and $H_{A}^{s}(\mathbb{R}^3,\mathbb{C})$ is endowed with the following norm
\begin{displaymath}
\|u\|_{H_{A}^{s}(\mathbb{R}^3,\mathbb{C})} = \left([u]_{s,A}^2 +
\|u\|_{L^2}^2\right)^{\frac{1}{2}}.
\end{displaymath}

We shall need the following embedding theorem, see 
\cite[Lemma 3.5]{da}.
\begin{proposition}\label{pro1}
The space $H_A^s(\mathbb{R}^3, \mathbb{C})$ is continuously embedded
in $L^\vartheta(\mathbb{R}^3, \mathbb{C})$ for all $\vartheta \in
[2, 2_s^\ast]$. Furthermore, the space $H_A^s(\mathbb{R}^3,
\mathbb{C})$ is continuously compactly embedded in $L^\vartheta(K,
\mathbb{C})$ for all $\vartheta \in [2, 2_s^\ast]$ and any compact
set $K \subset \mathbb{R}^3$.
\end{proposition}

Next, we state the diamagnetic inequality whose proof
can be found in d'Avenia and  Squassina \cite{da}.
\begin{lemma}\label{lem2.1}
Let $u\in H_A^s(\mathbb{R}^3).$
Then $|u|\in H^s(\mathbb{R}^3),$
that is
\begin{displaymath}
\||u|\|_{s}\leq \|u\|_{s,A}.
\end{displaymath}
\end{lemma}

By  \cite[Proposition 3.6 ]{di}, for all $u \in
H^{s}(\mathbb{R}^3)$, we have
\begin{displaymath}
[u]_{s} = \|(-\Delta)^{\frac{s}{2}}\|_{L^2(\mathbb{R}^3)},
\end{displaymath}
i.e.
\begin{displaymath}
\iint_{\mathbb{R}^{6}}\frac{|u(x)-u(y)|^2}{|x-y|^{3+2s}}dxdy  =
\int_{\mathbb{R}^{3}}|(-\Delta)^{\frac{s}{2}}u(x)|^2dx.
\end{displaymath}
Moreover,
\begin{displaymath}
\iint_{\mathbb{R}^{6}}\frac{(u(x)-u(y))(v(x)-v(y))}{|x-y|^{3+2s}}dx
dy  =
\int_{\mathbb{R}^{3}}(-\Delta)^{\frac{s}{2}}u(x)\cdot(-\Delta)^{\frac{s}{2}}v(x)
dx.
\end{displaymath}

For problem \eqref{e1.1}, we shall  use the Banach space $E$ defined
by
\begin{displaymath}
E =\left\{u\in H_{A}^{s}(\mathbb{R}^3,\mathbb{C}):
\int_{\mathbb{R}^3}V(x)|u|^{2}dx<\infty \right\}
\end{displaymath}
with the norm
\begin{displaymath}
\|u\|_E := \left([u]_{s,A}^2 +
\int_{\mathbb{R}^3}V(x)|u|^2dx\right)^{\frac{1}{2}}.
\end{displaymath}
It follows from the assumption $(\mathcal {V})$  that the embedding
$E \hookrightarrow H_{A}^{s}(\mathbb{R}^3,\mathbb{C})$ is
continuous. Moreover, the norm $\|\cdot\|_E$ is equivalent to the
norm  
\begin{displaymath}
\|u\|_\varepsilon := \left([u]_{s,A}^2 +
\varepsilon^{-2s}\int_{\mathbb{R}^3}V(x)|u|^2dx\right)^{\frac{1}{2}}
\
\hbox{for each}\
\varepsilon > 0.
\end{displaymath}

Obviously, for each $\theta \in [2, 2_s^\ast]$, there is
$c_{\theta}>0$ such that
\begin{equation}\label{e2.1}
|u|_{\theta} \leq c_{\theta}\|u\|_E \leq
c_{\theta}\|u\|_{\varepsilon},
\end{equation}
where $0 < \varepsilon < 1$. Hereafter, we shortly denote by
$\|\cdot\|_\nu$ the norm of Lebesgue space $L^\nu(\Omega)$ with
$\nu\geq1$.

Next, we state the following  Hardy--Littlewood--Sobolev inequality,
see Lieb and Loss  \cite[Theorem~4.3]{lie}.
\begin{lemma}\label{lem2.4}
Assume that  $p, \iota>1$ and $0<\mu <N$, $N \geq 3$ with
$1/p+(N-\mu)/N+1/\iota=2$, $f\in L^p(\mathbb{R}^N),$ and $h\in
L^\iota(\mathbb{R}^N)$. Then there exists a sharp constant,
$C(p,\iota,\mu,N)$ independent of $f,h$, such that
\begin{equation}\label{e2.2}
\int\int_{\mathbb{R}^{2N}}\frac{f(x)h(y)}{|x-y|^{N-\mu}}~dxdy \leq
C(p,\iota,\mu,N) \|f\|_{L^p} \|h\|_{L^\iota}.
\end{equation}
If we set $p=\iota=2N/(N+\mu)$, then
\begin{align*}
C(p,\iota,\mu,N)=C(N,\mu)=
\pi^{\frac{N-\mu}{2}}\frac{\Gamma(\frac{\mu}{2})}{\Gamma(\frac{N+\mu}{2})}\left\lbrace
\frac{\Gamma(\frac{N}{2})}{\Gamma(N)}\right\rbrace^{\frac{\mu}{N}}.
\end{align*}
\end{lemma}
If $u = v = |w|^q$,  Lemma \ref{lem2.4} implies that
\begin{displaymath}
\int_{\mathbb R^{N}}\left(\mathcal{I}_\mu*|w|^{q}\right)|w|^{q} \,dx
\end{displaymath}
is well defined if $w\in L^{\iota q}(\mathbb R^{N})$ for some $r >
1$ satisfying $2/r + (N-\mu)/N = 2$. Moreover, in the upper critical
case,
\begin{equation}\label{e2.3}
\int_{\mathbb
R^{N}}\left(\mathcal{I}_\mu*|u|^{2_s^\sharp}\right)|u|^{2_s^\sharp}
\,dx \leq C(N,\mu) \|u\|_{2_s^\ast}^{22_s^\sharp}
\end{equation}
and the equality holds if and only if
\begin{equation}\label{e2.4}
  u=C\left(\frac{l}{l^2+|x-m|^2}\right)^{\frac{N-2}{2}},
\end{equation}
for some $x_0 \in  \mathbb R^{N}$, where $C > 0$ and $l > 0$, see
\cite{lie}. Let
\begin{equation}\label{e2.5}
 S=\inf_{ u \in D^{s}(\mathbb{R}^N) \setminus \{0\}} \left\{
\iint_{\mathbb{R}^{2N}}\frac{|u(x)-u(y)|^2}{|x-y|^{N+2s}}dxdy:\;
\int_{\mathbb{R}^N}|u|^{2_s^{*}}dx=1\right\}
\end{equation}
and
\begin{equation}\label{e2.6}
S_{H} = \inf_{ u \in D^{s}(\mathbb{R}^N)\setminus \{0\}} \left\{
\iint_{\mathbb{R}^{2N}}\frac{|u(x)-u(y)|^2}{|x-y|^{N+2s}}dxdy:\;
\int_{\mathbb{R}^N}
\left(\mathcal{I}_\mu*|u|^{2_s^\sharp}\right)|u|^{2_s^\sharp}~dx=1
\right\}.
\end{equation}
By \eqref{e2.5} and \eqref{e2.3}, $S_{H}$ is achieved if and only
if $u$ satisfies \eqref{e2.4} and $S_{H} =
S/C(N,\mu)^{\frac{1}{p^\ast}},$ see Mukherjee and Sreenadh
\cite{mu1}.

%%%%%%%%%%%%%%%%%%%%%%%%%%%%%%%%%%%%%%%%%%%%%%%%%%%%%%%%%%%%%%%%%%%%%%%%%%%%%%
%%%%%%%%%%%%%%%%%%%%%%%%%%%%%%%%%%%%%%%%%%%%%%%%%%%%%%%%%%%%%%%%%%%%%%%%%%%%%%%%%%%
\section{Proof of $(PS)_c$}\label{sec3}

In this section, in order to overcome the lack of compactness caused
by the upper critical exponents, we intend to employ the second
concentration-compactness principle, see \cite{li2, PP, zhang1} for
more details.

Now, let $s, t \in (0, 1)$ such that $4s + 2t \geq 3.$
We can see
that
\begin{equation}\label{e3.1}
H^s(\mathbb{R}^3, \mathbb{R}) \hookrightarrow
L^{\frac{12}{3+2t}}(\mathbb{R}^3, \mathbb{R}).
\end{equation}
Then by \eqref{e3.1}, we have
\begin{displaymath}
\int_{\mathbb{R}^3}u^2vdx \leq
\|u\|_{L^\frac{12}{3+2t}(\mathbb{R}^3)}^2\|v\|_{2_t^\ast} \leq C
\|u\|_{H^s(\mathbb{R}^3, \mathbb{R})}^2\|v\|_{D^{t,2}(\mathbb{R}^3)}
\end{displaymath}
for all $u \in H^s(\mathbb{R}^3, \mathbb{R})$, where
\begin{displaymath}
\|v\|_{D^{t,2}(\mathbb{R}^3)}^2 =
\iint_{\mathbb{R}^{6}}\frac{|u(x)-u(y)|^2}{|x-y|^{3+2t}}dxdy.
\end{displaymath}
The Lax-Milgram Theorem implies that there exists a unique
$\psi_{|u|}^t$ such that $\psi_{|u|}^t \in D^{t,2}(\mathbb{R}^3,
\mathbb{R})$ such that
\begin{equation}\label{e3.2}
(-\Delta)^t\psi_{|u|}^t = |u|^2 \quad\mbox{in }\ \mathbb{R}^3.
\end{equation}
Therefore, we have
\begin{equation}\label{e3.3}
\psi_{|u|}^t(x) =
\pi^{-\frac{3}{2}}2^{-2t}\frac{\Gamma(3-2t)}{\Gamma(t)}\int_{\mathbb{R}^{3}}\frac{|u(y)|^2}{|x-y|^{3-2t}}dy,\quad\
\hbox{for all}\
x \in \mathbb{R}^3.
\end{equation}
Furthermore, \eqref{e3.3} is convergent at infinity since $|u|^2 \in
L^{\frac{6}{3+2t}}(\mathbb{R}^3, \mathbb{R})$.

Next, we collect some properties of $\psi_{|u|}^t$, which will be
used in this paper. The following proposition can be proved by using
similar arguments as  in  \cite{am3, am4}.
\begin{proposition}\label{pro2.1} Assume that $4s + 2t \geq 3$ holds. Then  for any $u\in E$, we have
\begin{itemize}
\item[(i)] $\psi_{|u|}^t: H^s(\mathbb{R}^3, \mathbb{R}) \rightarrow D^{t,2}(\mathbb{R}^3,
\mathbb{R})$ is continuous and maps bounded sets into bounded sets;

\item[(ii)] if $u_n\rightharpoonup u$ in $E$, then
$\psi_{|u_n|}^t \rightharpoonup\psi_{|u|}^t$ in
$D^{t,2}(\mathbb{R}^3, \mathbb{R})$;

\item[(iii)] $\psi_{|\alpha u|}^t = \alpha^2\psi_{|u|}^t$ for any $\alpha \in
\mathbb{R}$ and $\psi_{|u(\cdot+y)|}^t(x) = \psi_{|u|}^t(x+y)$;

\item[(iv)] $\psi_{|u|}^t \geq 0$ for all $u \in E$. Moreover,
\begin{displaymath}
\|\psi_{|u|}^t\|_{D^{t,2}(\mathbb{R}^3, \mathbb{R})} \leq
C\|u\|_{L^\frac{12}{3+2t}(\mathbb{R}^3)}^2 \leq C\|u\|_\varepsilon^2
\end{displaymath}
and
\begin{displaymath}
\int_{\mathbb{R}^3} \psi_{|u|}^t|u|^2dx \leq
C\|u\|_{L^\frac{12}{3+2t}(\mathbb{R}^3)}^4 \leq
C\|u\|_\varepsilon^4.
\end{displaymath}
\end{itemize}
\end{proposition}
We shall  use the following equivalent form
\begin{equation}\label{e3.4}
M\Big([u]_{s,A}^2\Big)(-\Delta)_{A}^su + \varepsilon^{-2s} V(x)u +
\varepsilon^{-2s}\psi_{|u|}^t u =
 \varepsilon^{-2s}k(x)|u|^{r-2}u + \varepsilon^{-2s}|u|^{2_s^\ast-2}u,
\end{equation}
for $x\in  \mathbb{R}^3$. Now,  the energy functional
$J_\varepsilon: E \rightarrow \mathbb{R}$ associated to
problem~\eqref{e1.1} is defined by
\begin{eqnarray}\label{e3.5}
J_\varepsilon(u) &:=&\nonumber \frac12 \mathcal
{M}\left([u]_{s,A}^2\right)+
\frac{\varepsilon^{-2s}}{2}\int_{\mathbb{R}^3}V(x)|u|^2dx +
\frac{\varepsilon^{-2s}}{4}\int_{\mathbb{R}^3}\psi_{|u|}^t|u|^2dx
\\
&& - \frac{\varepsilon^{-2s}}{r}\int_{\mathbb{R}^{3}}k(x)|u|^rdx
-\frac{\varepsilon^{-2s}}{22_s^\sharp}\int_{\mathbb
R^{3}}\left(\mathcal{I}_\mu*|u|^{2_s^\sharp}\right)|u|^{2_s^\sharp}\,dx.
\end{eqnarray}
Clearly, $J_\varepsilon$ is of class $C^1 (E, \mathbb{R})$ (see
\cite{w1}). Moreover, the Fr\'{e}chet derivative of $J_\varepsilon$
is given by
\begin{equation}\label{e3.6}\begin{aligned}
\langle J_\varepsilon'(u), v\rangle =&
M\left([u]_{s,A}^2\right)\mathscr{R}\iint_{\mathbb{R}^{6}}\frac{(u(x)-e^{i(x-y)\cdot
A(\frac{x+y}{2})}u(y))\overline{(v(x)-e^{i(x-y)\cdot
A(\frac{x+y}{2})}v(y))}}{|x-y|^{3+2s}}dx dy\\
&+ \varepsilon^{-2s}\mathscr{R}\int_{\mathbb{R}^3}V(x)u\bar{v}dx +
\varepsilon^{-2s}\mathscr{R}\int_{\mathbb{R}^3}\psi_{|u|}^tu\bar{v}dx
- \varepsilon^{-2s}\mathscr{R}
\int_{\mathbb{R}^3}k(x)|u|^{r-2}u\bar{v}
dx \\
& - \varepsilon^{-2s}\mathscr{R}
\int_{\mathbb{R}^3}\left(\mathcal{I}_\mu*|u|^{2_s^\sharp}\right)|u|^{2_s^\sharp-2}u\bar{v}
dx, \quad \forall\ u, v \in E.
\end{aligned}\end{equation}

\begin{lemma}\label{lem3.1}
Let conditions $(\mathcal {V})$  and $(\mathfrak{M})$ hold. Then for any $0 <
\varepsilon < 1$, $(PS)_{c}$ sequence $\{u_n\}_n$ for
$J_\varepsilon$
  is bounded in $E$ and $c \geq 0$.
\end{lemma}
\begin{proof}
Let sequence $\{u_n\}_n$ be a $(PS)_{c}$ sequence for
$J_\varepsilon$, that is $J_\varepsilon(u_n) \rightarrow c$ and
$J_\varepsilon'(u_n) \rightarrow 0$ in $E'$. It follows from
$(\mathcal {V})$  and $(\mathfrak{M})$ that
\begin{eqnarray}\label{e3.7}
c+ o(1)\|u_n\|_\varepsilon&=&\nonumber J_\varepsilon(u_n) -
\frac{1}{r}\langle J_\varepsilon'(u_n), u_n\rangle = \frac12\mathcal
{M}\left([u_n]_{s,A}^2\right)-
\frac{1}{p}M\left([u_n]_{s,A}^2\right)[u_n]_{s,A}^2 \\
&&\nonumber
\mbox{}+\left(\frac{1}{2}-\frac{1}{r}\right)\varepsilon^{-2s}
\int_{\mathbb{R}^3}V(x)|u_n|^2dx +
\left(\frac{1}{4}-\frac{1}{p}\right)\varepsilon^{-2s}
\int_{\mathbb{R}^3}\psi_{|u_n|}^t|u_n|^2dx\\
&&\nonumber \mbox{} +
\left(\frac{1}{r}-\frac{1}{22_s^\sharp}\right)\varepsilon^{-2s}\int_{\mathbb
R^{3}}\left(\mathcal{I}_\mu*|u_n|^{2_s^\sharp}\right)|u_n|^{2_s^\sharp}\,dx\\
&\geq&\nonumber
\left(\frac{1}{2\sigma}-\frac{1}{r}\right)M\left([u_n]_{s,A}^2\right)[u_n]_{s,A}^2
+ \left(\frac{1}{2}-\frac{1}{r}\right)\varepsilon^{-2s}
\int_{\mathbb{R}^3}V(x)|u_n|^2dx \\
&\geq&  \left(\frac{1}{2\sigma}-\frac{1}{r}\right)m_1
[u_n]_{s,A}^{2\sigma} +
\left(\frac{1}{2}-\frac{1}{r}\right)\varepsilon^{-2s}
\int_{\mathbb{R}^3}V(x)|u_n|^2dx.
\end{eqnarray}
This fact together with $2<r<2_s^\ast$ implies that $\{u_n\}_n$ is
bounded in $E$. Moreover, we can get $c\geq 0,$ by passing to the
limit in \eqref{e3.7}.
\end{proof}

\begin{lemma}\label{lem3.2} Let conditions $(\mathcal {V})$  and $(\mathfrak{M})$ hold.  Then  for any $0 < \varepsilon < 1$, the energy functional $J_\varepsilon$ satisfies $(PS)_c$ condition, for all $c \in
\left(0,\, \alpha_0\varepsilon^{\tau}\right)$, where
\begin{eqnarray}\label{e3.8}
\alpha_0 :=
\min\left\{\left(\frac{1}{r}-\frac{1}{22_s^\sharp}\right)\left(m_1S_H^\sigma\right)^{\frac{2_s^\sharp}{2_s^\sharp-\sigma}},\
\left(\frac{1}{2\sigma}-\frac{1}{r}\right)\left(m_1^{\frac{2_s^\sharp}{2\sigma}}\hat{C_\mu}^{-1}S^{\frac{2_s^\sharp}{2}}\right)^{\frac{2\sigma}{2_s^\sharp-2\sigma}}\right\}
\end{eqnarray}
and
\begin{eqnarray}\label{e3.9}
\tau := \max\left\{\frac{2s2_s^\sharp}{2_s^\sharp-\sigma},\
\frac{4\sigma s}{2_s^\sharp-2\sigma}\right\}.
\end{eqnarray}
\end{lemma}
\begin{proof}
If $\inf_{n\in\mathbb N}\|u_{n}\|_{\varepsilon}= 0$, then there
exists a subsequence of $\{u_{n}\}$ such that $u_{n} \rightarrow 0$
in $E$ as $n \rightarrow \infty$. Thus, we assume that
$\inf_{n\in\mathbb N}\|u_{n}\|_{\varepsilon}=d_1>0$ in the
sequel.

By Lemma \ref{lem3.1}, we know  that $\{u_n\}_n$ is bounded in
$E$. Thus, by the diamagnetic inequality, $\{|u_n|\}_n$ is bounded in
$H^s(\mathbb{R}^3)$.  Furthermore, we have $u_n\rightarrow u$ a.e.
in $\mathbb{R}^3$ and $u_n \rightharpoonup u$ in $E$. Let
\begin{eqnarray*}
|(-\Delta)^{\frac{s}{2}} u_n|^2 \rightharpoonup \omega, \ \
|u_n|^{2_s^\ast}\rightharpoonup \xi
\end{eqnarray*}
and
\begin{eqnarray*}
\left(\mathcal{I}_\mu*|u_n|^{2_s^\sharp}\right)|u_n|^{2_s^\sharp}\rightharpoonup
\nu \ \ \mbox{weakly\ in\ the\ sense\ of\ measures,}
\end{eqnarray*}
where $\omega$, $\xi$  and $\nu$ are bounded nonnegative measures on
$\mathbb{R}^3$. Then by using  the fractional version of
concentration compactness principle in the fractional Sobolev space
(see \cite{li2}), up to a subsequence, there exists a (at most
countable) set of distinct points $\{x_i\}_{i\in I} \subset
\mathbb{R}^3$ and a family of positive numbers $\{\nu_i\}_{i\in I}$
such that
\begin{eqnarray}\label{e3.10}
\nu = \left(\mathcal{I}_\mu*|u|^{2_s^\sharp}\right)|u|^{2_s^\sharp}
+ \sum_{i \in I} \nu_i\delta_{x_i},\quad \sum_{i \in I}
\nu_i^{\frac{3}{3+\mu}} < \infty,
\end{eqnarray}
\begin{eqnarray}\label{e3.11}
\xi \geq |u|^{2_s^\ast} + C_\mu^{-\frac{3}{3+\mu}}\sum_{i \in I}
\nu_i^{\frac{3}{3+\mu}}\delta_{x_i},\quad \xi_i \geq C_\mu
^{-\frac{3}{3+\mu}}\nu_i^{\frac{3}{3+\mu}}
\end{eqnarray}
and
\begin{eqnarray}\label{e3.12}
\omega \geq |(-\Delta)^{\frac{s}{2}} u|^2 + S_H\sum_{i \in I}
\nu_i^{\frac{1}{2_s^\sharp}}\delta_{x_i}, \quad  \omega_i \geq
 S_H\nu_i^{\frac{1}{2_s^\sharp}},
\end{eqnarray}
where $\delta_{x_i}$ is the Dirac-mass of mass 1 concentrated at $x
\in \mathbb{R}^3$.

Now, let $i \in I.$
We claim that  either  $\nu_i = 0$ or
\begin{eqnarray}\label{e3.13}
\nu_i \geq
\left(m_1S_H^\sigma\right)^{\frac{2_s^\sharp}{2_s^\sharp-\sigma}}\varepsilon^{\frac{2s2_s^\sharp}{2_s^\sharp-\sigma}}.
\end{eqnarray}
In order to prove \eqref{e3.13}, we take $\phi \in
C_0^\infty(\mathbb{R}^3)$  satisfying $0 \leq \phi \leq 1$; $\phi
\equiv 1$ in $B(x_i, \epsilon)$, $\phi(x) = 0$ in $\mathbb{R}^3
\setminus B(x_i, 2\epsilon)$. For any $\epsilon
> 0$, define $\phi_\epsilon :=
\phi\left(\frac{x-x_i}{\epsilon}\right)$, where $i \in I$. Clearly,
$\{\phi_\epsilon u_{n}\}$ is bounded in $E$ and $\langle
J_\varepsilon'(u_n), u_n\phi_\epsilon)\rangle \rightarrow 0$ as $n
\rightarrow \infty$. Hence
\begin{eqnarray}\label{e3.14}
&&\nonumber
M\left(\|u_{n}\|_{s,A}^2\right)\left(\iint_{\mathbb{R}^{6}}\frac{|u_{n}(x)-e^{i(x-y)\cdot
A(\frac{x+y}{2})}u_n(y)|^2\phi_\epsilon(y)}{|x-y|^{3+2s}}dx dy +
\varepsilon^{-2s}\int_{\mathbb{R}^3}V(x)|u_{n}|^2\phi_\epsilon(x) dx\right)\\
&&\mbox{}\nonumber = - \mathscr{R}\left\{
M\left(\|u_{n}\|_{s,A}^2\right)\iint_{\mathbb{R}^{6}}\frac{(u_{n}(x)-e^{i(x-y)\cdot
A(\frac{x+y}{2})}u_{n}(y))\overline{u_{n}(x)(\phi_\epsilon(x)-\phi_\epsilon(y))}}{|x-y|^{3+2s}}dx dy\right\}\\
&&\mbox{}\ \  + \varepsilon^{-2s}\int_{\mathbb
R^{3}}\left(\mathcal{I}_\mu*|u_{n}|^{2_s^\sharp}\right)|u_{n}|^{2_s^\sharp}\phi_\epsilon\,dx
 + \varepsilon^{-2s}\int_{\mathbb{R}^3}k(x)|u_{n}|^r\phi_\epsilon dx +o_n(1).
\end{eqnarray}
We deduce from $(M_2)$ and the diamagnetic inequality  that
\begin{eqnarray}\label{e3.15}
&&\nonumber
M\left(\|u_{n}\|_{s,A}^2\right)\left(\iint_{\mathbb{R}^{6}}\frac{|u_{n}(x)-e^{i(x-y)\cdot
A(\frac{x+y}{2})}u_n(y)|^2\phi_\epsilon(y)}{|x-y|^{3+2s}}dx dy +
\int_{\mathbb{R}^3}V(x)|u_{n}|^2\phi_\epsilon(x) dx\right)\\
&&\nonumber \geq m_1
\left(\iint_{\mathbb{R}^{6}}\frac{|u_{n}(x)-e^{i(x-y)\cdot
A(\frac{x+y}{2})}u_{n}(y)|^2\phi_\epsilon(y)}{|x-y|^{3+2s}}dx dy +
\int_{\mathbb{R}^3}V(x)|u_{n}|^2\phi_\epsilon(x) dx\right)^{\sigma}\\
&&\geq m_1
\left(\iint_{\mathbb{R}^{6}}\frac{\left||u_{n}(x)|-|u_{n}(y)|\right|^2\phi_\epsilon(y)}{|x-y|^{3+2s}}dxdy\right)^{\sigma}.
\end{eqnarray}
Notice that
\begin{eqnarray}\label{e3.16}
\lim_{\epsilon \rightarrow 0}\lim_{n \rightarrow
\infty}\iint_{\mathbb{R}^{6}}\frac{\left||u_{n}(x)|-|u_{n}(y)|\right|^2\phi_\epsilon(y)}{|x-y|^{3+2s}}dxdy
= \lim_{\epsilon \rightarrow 0}\int_{\mathbb{R}^{3}}\phi_\epsilon
d\omega = \omega_i
\end{eqnarray}
and
\begin{eqnarray}\label{e3.17}
\lim_{\epsilon \rightarrow 0}\lim_{n \rightarrow \infty}
\int_{\mathbb
R^{3}}\left(\mathcal{I}_\mu*|u_{n}|^{2_s^\sharp}\right)|u_{n}|^{2_s^\sharp}\phi_\epsilon\,dx
= \lim_{\epsilon \rightarrow 0}\int_{\mathbb{R}^{3}}\phi_\epsilon
d\nu = \nu_i.
\end{eqnarray}
From the  H\"{o}lder inequality, we have
\begin{eqnarray}\label{e3.18}
&&\nonumber\left|\mathscr{R}\left\{
M\left(\|u_{n}\|_{s,A}^2\right)\iint_{\mathbb{R}^{6}}\frac{(u_{n}(x)-e^{i(x-y)\cdot
A(\frac{x+y}{2})}u_{n}(y))\overline{u_{n}(x)(\phi_\epsilon(x)-\phi_\epsilon(y))}}{|x-y|^{3+2s}}dx dy\right\}\right|\\
&&\nonumber \mbox{} \ \leq
C\iint_{\mathbb{R}^{6}}\frac{|u_{n}(x)-e^{i(x-y)\cdot
A(\frac{x+y}{2})}u_{n}(y)|\cdot|\phi_\epsilon(x)-\phi_\epsilon(y)|\cdot|u_{n}(x)|}{|x-y|^{3+2s}}dxdy
\\
&& \mbox{} \  \leq C
\left(\iint_{\mathbb{R}^{6}}\frac{|u_{n}(x)|^2|\phi_\epsilon(x)-\phi_\epsilon(y)|^2}{|x-y|^{3+2s}}dxdy\right)^{1/2}.
\end{eqnarray}
\indent  As in the proof of Lemma 3.4 in Zhang et al. \cite{zhang2}, we
get
\begin{eqnarray}\label{e3.19}
\lim_{\epsilon\rightarrow
 0}\lim_{n\rightarrow\infty}\iint_{\mathbb{R}^{6}}\frac{|u_{n}(x)|^2|\phi_\epsilon(x)-\phi_\epsilon(y)|^2}{|x-y|^{3+2s}}dxdy
=0.
\end{eqnarray}
Since $\phi_\epsilon$ has compact support, by the definition of
$\phi_{\epsilon}(x)$, we have
\begin{eqnarray}\label{e3.20}
\lim_{\epsilon\rightarrow
0}\lim_{n\rightarrow\infty}\int_{\mathbb{R}^3}k(x)|u_{n}|^r\phi_\epsilon
dx =0.
\end{eqnarray}
Combining \eqref{e3.15}--\eqref{e3.20}, we get that
$$\varepsilon^{-2s}\nu_i \geq m_1\omega_i^\sigma.$$ It follows from
\eqref{e3.12}  that $\nu_i = 0$ or
\begin{eqnarray*}
 \nu_i \geq
\left(m_1S_H^\sigma\right)^{\frac{2_s^\sharp}{2_s^\sharp-\sigma}}\varepsilon^{\frac{2s2_s^\sharp}{2_s^\sharp-\sigma}}.
\end{eqnarray*}

Next, we shall prove that $\nu_i = 0$, $ \ \hbox{for all} \ \ i \in I$ and
$\nu_\infty = 0$.

Indeed, if not,  then there exists a $i\in I$ such that
\eqref{e3.13} holds.  Similar to \eqref{e3.7}, we deduce
\begin{align}\label{e3.21}
  c &=\nonumber\lim_{\epsilon\rightarrow 0}  \lim_{ n \rightarrow
\infty} \left(J_\varepsilon(u_n) -
\frac{1}{r}\langle J_\varepsilon'(u_n), u_n\rangle\right)\\
&\geq\nonumber
\left(\frac{1}{r}-\frac{1}{22_s^\sharp}\right)\varepsilon^{-2s}\int_{\mathbb
R^{3}}\left(\mathcal{I}_\mu*|u|^{2_s^\sharp}\right)|u|^{2_s^\sharp}\,dx\\
&\geq\nonumber
\left(\frac{1}{r}-\frac{1}{22_s^\sharp}\right)\varepsilon^{-2s}\int_{\mathbb
R^{3}}\left(\mathcal{I}_\mu*|u|^{2_s^\sharp}\right)|u|^{2_s^\sharp}\phi_\epsilon\,dx\\
&\geq\left(\frac{1}{p}-\frac{1}{22_s^\sharp}\right)\varepsilon^{-2s}
\nu_i \geq
\left(\frac{1}{r}-\frac{1}{22_s^\sharp}\right)\left(m_1S_H^\sigma\right)^{\frac{2_s^\sharp}{2_s^\sharp-\sigma}}\varepsilon^{\frac{2s\sigma}{2_s^\sharp-\sigma}}.
\end{align}
For the concentration at infinity, letting $R > 0$, we take a cut
off function $\phi_R \in C^\infty(\mathbb{R}^3)$ such that
\begin{equation*}
\phi_R(x) = \begin{cases}
    0 &\quad |x| < R,\\
    1 &\quad |x| > R+1.
    \end{cases}
\end{equation*}
Define
\begin{eqnarray*}
\omega_\infty = \lim\limits_{R\rightarrow
\infty}\limsup\limits_{n\rightarrow\infty}\displaystyle\int_{\{x \in
\mathbb{R}^3: |x|>R\}}|(-\Delta)^{\frac{s}{2}} u_n|^{2}dx,
\end{eqnarray*}
\begin{eqnarray*}
\xi_\infty = \lim\limits_{R\rightarrow
\infty}\limsup\limits_{n\rightarrow\infty}\displaystyle\int_{\{x \in
\mathbb{R}^3: |x|>R\}}|u_n|^{2_s^\ast}dx
\end{eqnarray*} and
\begin{eqnarray*}
\nu_\infty = \lim\limits_{R\rightarrow
\infty}\limsup\limits_{n\rightarrow\infty}\displaystyle\int_{\{x \in
\mathbb{R}^3:
|x|>R\}}\left(\mathcal{I}_\mu*|u_n|^{2_s^\sharp}\right)|u_n|^{2_s^\sharp}dx.
\end{eqnarray*}
By using  the fractional version of concentration compactness
principle (see \cite{li2}), for the energy at infinity, we have
\begin{eqnarray}\label{e3.22}
\limsup\limits\limits_{n\rightarrow\infty}\displaystyle\int_{\mathbb{R}^3}\left(\mathcal{I}_\mu*|u_n|^{2_s^\sharp}\right)|u_n|^{2_s^\sharp}dx
= \displaystyle\int_{\mathbb{R}^3}d\nu + \nu_\infty,
\end{eqnarray}
\begin{eqnarray}\label{e3.23}
\limsup\limits_{n\rightarrow\infty}\displaystyle\int_{\mathbb{R}^3}|(-\Delta)^{\frac{s}{2}}
u_n|^{2}dx = \int_{\mathbb{R}^3}d\omega + \omega_\infty,
\end{eqnarray}
\begin{eqnarray}\label{e3.24}
\limsup\limits\limits_{n\rightarrow\infty}\displaystyle\int_{\mathbb{R}^3}|u_n|^{2_s^\ast}dx
= \displaystyle\int_{\mathbb{R}^3}d\xi + \xi_\infty,
\end{eqnarray}
\begin{eqnarray}\label{e3.25}
\xi_\infty \leq
\left(S^{-1}\omega_\infty\right)^{\frac{2_s^\ast}{2}},
\end{eqnarray}
\begin{eqnarray}\label{e3.26}
\nu_\infty \leq C_\mu\left(\int_{\mathbb{R}^3}d\xi +
\xi_\infty\right)^{\frac{3+\mu}{6}}\xi_\infty^{\frac{3+\mu}{6}}
\end{eqnarray}
and
\begin{eqnarray}\label{e3.27}
\nu_\infty \leq S_H^{-2_s^\sharp}\left(\int_{\mathbb{R}^3}d\omega +
\omega_\infty\right)^{\frac{2_s^\sharp}{2}}\omega_\infty^{\frac{2_s^\sharp}{2}}.
\end{eqnarray}
By using the Hardy-Littlewood-Sobolev and H\"older's inequality, we
get
\begin{align}\label{e3.28}
\nu_\infty &=\nonumber   \lim_{R\rightarrow \infty}  \lim_{ n
\rightarrow
\infty}\displaystyle\int_{\mathbb{R}^3}\left(\mathcal{I}_\mu*|u_{n}|^{2_s^\sharp}\right)|u_{n}|^{2_s^\sharp}\phi_R(y)dx  \\
&\nonumber \leq C_\mu\lim_{R\rightarrow \infty}  \lim_{ n
\rightarrow
\infty}|u_{n}|_{2_s^\ast}^{2_s^\sharp}\left(\int_{\mathbb{R}^3}|u_{n}(x)|^{2_s^*}\phi_R(y)dx\right)^{\frac{2_s^\sharp}{2_s^\ast}}\\
& \leq
\hat{C_\mu}\xi_\infty^{\frac{2_s^\sharp}{2_s^\ast}}.\end{align}
Note that  $\{\phi_R u_{n}\}$ is also bounded in $E$.  Hence, $\langle
J_{\varepsilon}'(u_n), u_n\phi_R)\rangle \rightarrow 0$ as $n
\rightarrow \infty$, which yields that
\begin{eqnarray}\label{e3.29}
&&\nonumber
M\left(\|u_{n}\|_{s,A}^2\right)\left(\iint_{\mathbb{R}^{6}}\frac{|u_{n}(x)-e^{i(x-y)\cdot
A(\frac{x+y}{2})}u_{n}(y)|^2\phi_R(y)}{|x-y|^{3+2s}}dx dy +
\varepsilon^{-2s}\int_{\mathbb{R}^3}V(x)|u_{n}|^2\phi_R(x) dx\right)\\
&&\mbox{}\nonumber = - \mathscr{R}\left\{
M\left(\|u_{n}\|_{s,A}^2\right)\iint_{\mathbb{R}^{6}}\frac{(u_{n}(x)-e^{i(x-y)\cdot
A(\frac{x+y}{2})}u_{n}(y))\overline{u_{n}(x)(\phi_R(x)-\phi_R(y))}}{|x-y|^{3+2s}}dx dy\right\}\\
&&\mbox{}\ \  + \varepsilon^{-2s}\int_{\mathbb
R^{3}}\left(\mathcal{I}_\mu*|u_{n}|^{2_s^\sharp}\right)|u_{n}|^{2_s^\sharp}\phi_R\,dx
+  \varepsilon^{-2s}\int_{\mathbb{R}^3}k(x)|u_{n}|^r\phi_R dx
+o_n(1).
\end{eqnarray}
It's easy to get
\begin{eqnarray*}
\limsup\limits_{R\rightarrow\infty}\limsup\limits_{n\rightarrow\infty}\iint_{\mathbb{R}^{6}}\frac{||u_{n}(x)|-|u_{n}(y)||^2\phi_R(y)}{|x-y|^{3+2s}}dxdy
= \omega_\infty
\end{eqnarray*}
and
\begin{eqnarray*}
&& \left|\mathscr{R}\left\{
M\left(\|u_{n}\|_{s,A}^2\right)\iint_{\mathbb{R}^{6}}\frac{(u_{n}(x)-e^{i(x-y)\cdot
A(\frac{x+y}{2})}u_{n}(y))\overline{u_{n}(x)(\phi_R(x)-\phi_R(y))}}{|x-y|^{3+2s}}dx
dy\right\}\right| \\
&& \mbox{}  \leq
C\left(\iint_{\mathbb{R}^{6}}\frac{|u_{n}(x)|^2|\phi_R(x)-\phi_R(y)|^2}{|x-y|^{3+2s}}dxdy\right)^{1/2}.
\end{eqnarray*}
Furthermore,
\begin{eqnarray*}
&& \limsup\limits_{R \rightarrow \infty}\limsup\limits_{n
\rightarrow \infty}
\iint_{\mathbb{R}^{6}}\frac{|u_{n}(x)|^2|\phi_R(x)-\phi_R(y)|^2}{|x-y|^{3+2s}}dxdy
\\
&& \mbox{} = \limsup\limits_{R \rightarrow \infty}\limsup\limits_{n
\rightarrow \infty}
\iint_{\mathbb{R}^{6}}\frac{|u_{n}(x)|^2|(1-\phi_R(x))-(1-\phi_R(y))|^2}{|x-y|^{3+2s}}dxdy.
\end{eqnarray*}
By the proof of Lemma 3.4 in Zhang et al. \cite{zhang2}, we have
\begin{eqnarray*}
\limsup\limits_{R \rightarrow \infty}\limsup\limits_{n \rightarrow
\infty}
\iint_{\mathbb{R}^{6}}\frac{|u_{n}(x)|^2|(1-\phi_R(x))-(1-\phi_R(y))|^2}{|x-y|^{3+2s}}dxdy
= 0.
\end{eqnarray*}
It follows from $(M_2)$ that
\begin{eqnarray*}
&&
M\left(\|u_{n}\|_{s,A}^2\right)\left(\iint_{\mathbb{R}^{6}}\frac{|u_{n}(x)-e^{i(x-y)\cdot
A(\frac{x+y}{2})}u_{n}(y)|^2\phi_R(y)}{|x-y|^{3+2s}}dx dy +
\int_{\mathbb{R}^3} |u_{n}|^2\phi_R(x) dx\right)\\
&& \mbox{} \ \ \geq  m_1
\left(\iint_{\mathbb{R}^{6}}\frac{|u_{n}(x)-e^{i(x-y)\cdot
A(\frac{x+y}{2})}u_{n}(y)|^2\phi_R(y)}{|x-y|^{3+2s}}dx dy +
\int_{\mathbb{R}^3} |u_{n}|^2\phi_R(x) dx\right)^{\sigma}\\
&& \mbox{} \ \ \geq m_1
\left(\iint_{\mathbb{R}^{6}}\frac{\left||u_n(x)|-|u_n(y)|\right|^2\phi_R(y)}{|x-y|^{3+2s}}dxdy\right)^{\sigma}
= m_1 \omega_\infty^{\sigma}.
\end{eqnarray*}
By the definition of $\phi_R$, we have
\begin{eqnarray*}
\lim_{R \rightarrow \infty}\lim_{n \rightarrow
\infty}\int_{\mathbb{R}^3}k(x)|u_{n}|^r\phi_R dx \leq k^\ast\lim_{R
\rightarrow \infty}\lim_{n \rightarrow
\infty}\int_{\mathbb{R}^3}|u_{n}|^r\phi_R dx = 0.
\end{eqnarray*}
Therefore, by \eqref{e3.29} together with \eqref{e3.28}, we can
obtain that
$$\varepsilon^{-2s}\hat{C_\mu}\xi_\infty^{\frac{2_s^\sharp}{2_s^\ast}} \geq  \varepsilon^{-2s}\nu_\infty \geq m_1\omega_\infty^\sigma.$$ It follows from \eqref{e3.25}  that $\omega_\infty = 0$ or
\begin{eqnarray}\label{e3.30}
\omega_\infty \geq
\left(m_1\hat{C_\mu}^{-1}S^{\frac{2_s^\sharp}{2}}\right)^{\frac{2}{2_s^\sharp-2\sigma}}\varepsilon^{\frac{4s}{2_s^\sharp-2\sigma}}.
\end{eqnarray}
If \eqref{e3.30} holds, then we have
\begin{align}\label{e3.31}
  c &=\nonumber\lim_{R\rightarrow \infty}  \lim_{ n \rightarrow
\infty} \left(J_\varepsilon(u_n) -  \frac{1}{r}\langle
J_\varepsilon'(u_n), u_n) \rangle\right)\\
&\geq\nonumber \left(\frac{1}{2\sigma}-\frac{1}{r}\right)m_1
\left(\iint_{\mathbb{R}^{6}}\frac{\left||u_n(x)|-|u_n(y)|\right|^2\phi_R(y)}{|x-y|^{3+2s}}dxdy\right)^{\sigma}
\\
&\geq
\left(\frac{1}{2\sigma}-\frac{1}{r}\right)m_1\omega_\infty^\sigma
\geq
\left(\frac{1}{2\sigma}-\frac{1}{r}\right)\left(m_1^{\frac{2_s^\sharp}{2\sigma}}\hat{C_\mu}^{-1}S^{\frac{2_s^\sharp}{2}}\right)^{\frac{2\sigma}{2_s^\sharp-2\sigma}}\varepsilon^{\frac{4\sigma
s}{2_s^\sharp-2\sigma}}.
\end{align}
Due to the selection of $\alpha_0$ and $\tau$,  for any $c <
\alpha_0\varepsilon^\tau$, this gives a contradiction. Thus,
$\omega_\infty = 0$. By  \eqref{e3.27}, we know that
$$\nu_i =0\ \mbox{ for all $i \in I$\quad and} \quad \nu_\infty = 0.$$
Thus
\begin{eqnarray} \label{e3.32}
\int_{\mathbb
R^{3}}\left(\mathcal{I}_\mu*|u_n|^{2_s^\sharp}\right)|u_n|^{2_s^\sharp}\,dx
\rightarrow \int_{\mathbb
R^{3}}\left(\mathcal{I}_\mu*|u|^{2_s^\sharp}\right)|u|^{2_s^\sharp}\,dx
\quad\mbox{as}\ n  \rightarrow \infty.
\end{eqnarray}
By the Br\'{e}zis--Lieb Lemma \cite{bre}, we get
\begin{eqnarray*}
\int_{\mathbb
R^{3}}\left(\mathcal{I}_\mu*|u_n-u|^{2_s^\sharp}\right)|u_n-u|^{2_s^\sharp}\,dx
\rightarrow 0 \quad \mbox{as}\  n \rightarrow \infty.
\end{eqnarray*}
Finally, with the aid of the weak lower semicontinuity of the norm,
condition $(m_1)$ and the Br\'{e}zis--Lieb Lemma \cite{bre}, we can
obtain that
\begin{eqnarray*}
 o(1)\|u_n\|&=& \nonumber\langle J_\varepsilon'(u_n), u_n\rangle = M\left([u_n]_{s,A}^2\right)[u_n]_{s,A}^2
+ \varepsilon^{-2s}\int_{\mathbb{R}^3} V(x)|u_n|^2dx \\
&&\mbox{}  + \varepsilon^{-2s}
\int_{\mathbb{R}^3}\psi_{|u_n|}^t|u_n|^2dx -
\varepsilon^{-2s}\int_{\mathbb
R^{3}}\left(\mathcal{I}_\mu*|u_n|^{2_s^\sharp}\right)|u_n|^{2_s^\sharp}\,dx
- \varepsilon^{-2s}\int_{\mathbb{R}^3}k(x)|u_n|^rdx\\
&\geq& m_0\left([u_n]_{s,A}^2 - [u]_{s,A}^2\right)
+ \varepsilon^{-2s}\int_{\mathbb{R}^3} V(x)(|u_n|^2-|u|^2)dx + M\left([u]_{s,A}^2\right)[u]_{s,A}^2\\
&&\mbox{}+ \varepsilon^{-2s}\int_{\mathbb{R}^3} V(x)|u|^2dx +
\varepsilon^{-2s} \int_{\mathbb{R}^3}\psi_{|u|}^t|u|^2dx\\
&&\mbox{}   -
 \varepsilon^{-2s}\int_{\mathbb
R^{3}}\left(\mathcal{I}_\mu*|u|^{2_s^\sharp}\right)|u|^{2_s^\sharp}\,dx
-
\varepsilon^{-2s}\int_{\mathbb{R}^3}k(x)|u|^rdx\\
&\geq& \min\{m_0,1\}\|u_n - u\|_\varepsilon^2  +
o(1)\|u\|_\varepsilon.
\end{eqnarray*}
Here we use the fact that $J_\varepsilon'(u) = 0$. This fact implies
that $\{u_n\}_n$ strongly converges to $u$ in $E$. Hence the proof of Lemma \ref{lem3.2}
is complete.
\end{proof}

%%%%%%%%%%%%%%%%%%%%%%%%%%%%%%%%%%%%%%%%%%%%%%%%%%%%%%%%%%%%%%%%%%%%%%%%%%%%%%%%%%%
\section{Proofs of main results}\label{sec4}

\indent In order to prove Theorems \ref{the1.1} and \ref{the1.2}, we
first prove that functional $J_\varepsilon(u)$ satisfies  the
following mountain pass geometry.
\begin{lemma}\label{lem4.1} Let
conditions  $(\mathcal {V})$  and $(\mathfrak{M})$ hold. Then for any $0 <
\varepsilon < 1$,
\begin{enumerate}
\item[$(C_1)$] There exist $\beta_\varepsilon, \rho_\varepsilon
> 0$ such that $J_\varepsilon(u) > 0$ if $u \in B_{\rho_\varepsilon}\setminus\{0\}$ and $J_\varepsilon(u) \geq \beta_\varepsilon$
if $u \in \partial B_{\rho_\varepsilon}$, where
$B_{\rho_\varepsilon} = \{u\in E: \|u\|_\varepsilon \leq
\rho_\varepsilon\}$;

\item[$(C_2)$] For any finite-dimensional subspace $H\subset E$,
\begin{displaymath}
J_\varepsilon(u) \rightarrow -\infty\quad \mbox{as}\quad \ u\in H, \
\|u\|_\varepsilon \rightarrow \infty.
\end{displaymath}
\end{enumerate}
\end{lemma}
\begin{proof} By the Hardy-Littlewood-Sobolev inequality and the Sobolev embedding theorem, we have
\begin{eqnarray*}
J_\varepsilon(u) &\geq&
\min\left\{\frac{m_1}{2\sigma},\frac{1}{2}\right
\}\|u\|_\varepsilon^2 - \varepsilon^{-2s}k^\ast|u|_r^r
-\varepsilon^{-2s}\int_{\mathbb
R^{3}}\left(\mathcal{I}_\mu*|u|^{2_s^\sharp}\right)|u|^{2_s^\sharp}\,dx\\
&\geq&\min\left\{\frac{m_1}{2\sigma},\frac{1}{2}\right
\}\|u\|_\varepsilon^2 - \varepsilon^{-2s}C\|u\|_\varepsilon^r
-\varepsilon^{-2s}C_\mu C\|u\|_\varepsilon^{22_s^\sharp}.
\end{eqnarray*}
Since $r>2$ and $22_s^\sharp > 2$, we can obtain   the conclusion
$(C_1)$ in Lemma \ref{lem4.1}.

On the other hand, by $(M_2)$, we have that
\begin{equation}\label{e4.1}
\mathcal {M}(t)\leq \mathcal {M}(1)t^{\sigma}\quad \text{for all
}t\geq 1.
\end{equation}
Let $v_0 \in C_0^\infty(\mathbb R^{3}, \mathbb{C})$ with
$\|v_0\|_\varepsilon = 1$. Thus,  we have
\begin{displaymath}
J_\varepsilon(tv_0) \leq \mathcal {M}(1)t^{2\sigma}  +
\frac{1}{2}t^{2}  + \frac{\varepsilon^{-2s} }{4}Ct^{4}
-\varepsilon^{-2s}t^{22_s^\sharp}\int_{\mathbb
R^{3}}\left(\mathcal{I}_\mu*|v_0|^{2_s^\sharp}\right)|v_0|^{2_s^\sharp}\,dx
- \varepsilon^{-2s}k_\ast t^r|v_0|_r^r.
\end{displaymath}
Note that all norms in a finite-dimensional space $H$ are equivalent
and $\max\{4, 2\sigma\} < 22_s^\sharp$, so we can also obtain   the
conclusion $(C_2)$ of Lemma \ref{lem4.1}.
\end{proof}

What we need to point out is that $J_\varepsilon(u)$ does not
satisfy $(PS)_c$ condition for any $c > 0$. Thus, we need to
construct  a special finite-dimensional subspace by which we
construct sufficiently small minimax levels.

On the one hand, by  \cite[Lemma 3.5]{bin}, we know that
\begin{displaymath}
\inf\left\{\iint_{\mathbb{R}^{6}}\frac{|\phi(x)-\phi(y)|^2}{|x-y|^{3+2s}}dxdy:
\phi \in C_0^\infty (\mathbb{R}^3), |\phi|_r = 1\right\} = 0.
\end{displaymath}
Thus, for any $1 > \delta > 0$ one can choose $\phi_\delta \in
C_0^\infty (\mathbb{R}^3)$ with $|\phi_\delta|_r = 1$ and
supp\,$\phi_\delta \subset B_{r_\delta} (0)$ so that
\begin{displaymath}
\iint_{\mathbb{R}^{6}}\frac{|\phi_\delta(x)-\phi_\delta(y)|^2}{|x-y|^{3+2s}}dxdy
\leq C\delta^{\frac{6-(3-2s)r}{r}}.
\end{displaymath}
Let
\begin{equation}\label{e4.2}
q_\delta(x) = e^{iA(0)x}\phi_\delta(x)
\end{equation}
and
\begin{equation}\label{e4.3}
q_{\varepsilon,\delta}(x) =
q_\delta(\varepsilon^{-\frac{\tau+2s}{3}}x),
\end{equation}
where $\tau$ is defined by \eqref{e3.9}.

On the other hand, since $22_s^\sharp>\sigma$,  there exists a
finite number $t_0 \in [0, +\infty)$ such that
\begin{eqnarray*}
\max_{t \geq 0}\mathcal {I}_\varepsilon(tq_{\delta}) &\leq&
\frac{C_0}{2}t_0^{2\sigma}\left(\iint_{\mathbb{R}^{6}}\frac{|q_{\delta}(x)-e^{i(x-y)\cdot
A(\frac{\varepsilon x+\varepsilon
y}{2})}q_{\delta}(y)|^2}{|x-y|^{3+2s}}dx dy\right)^{2\sigma} \\
&&\mbox{} + \frac{t_0^2}{2}\int_{\mathbb{R}^3}V\left(\varepsilon
x\right)|q_\delta|^2dx + \frac{t_0^4}{4}
\int_{\mathbb{R}^3}\psi_{|q_{\delta}|}^{t_0}||q_{\delta}|^2dx-
k_\ast\int_{\mathbb{R}^3} |q_\delta|^rdx :=
I_\varepsilon(t_0q_\delta),
\end{eqnarray*}
where
\begin{eqnarray*}
\mathcal {I}_\varepsilon(u) &:=&\nonumber
\frac{C_0}{2}\left(\iint_{\mathbb{R}^{6}}\frac{|u(x)-e^{i(x-y)\cdot
A(\frac{\varepsilon x+\varepsilon y}{2})}u(y)|^2}{|x-y|^{3+2s}}dx
dy\right)^{2\sigma} +
\frac{1}{2}\int_{\mathbb{R}^3}V\left(\varepsilon
x\right)|u|^2dx\\
&& \mbox{}  + \frac{1}{4} \int_{\mathbb{R}^3}\psi_{|u|}^t||u|^2dx -
k_\ast\int_{\mathbb{R}^3} |u|^rdx-\frac{1}{22_s^\sharp}\int_{\mathbb
R^{3}}\left(\mathcal{I}_\mu*|u|^{2_s^\sharp}\right)|u|^{2_s^\sharp}\,dx.
\end{eqnarray*}
Therefore, for any $t>0$ we get
\begin{eqnarray*}
J_\varepsilon(tq_{\varepsilon,\delta})
&\leq&\frac{C_0}{2}t^{2\sigma}\left(\iint_{\mathbb{R}^{6}}\frac{|q_{\varepsilon,\delta}(x)-e^{i(x-y)\cdot
A(\frac{x+y}{2})}q_{\varepsilon,\delta}(y)|^2}{|x-y|^{3+2s}}dx dy\right)^{2\sigma} \\
&& \mbox{} + \frac{t^2}{2}\varepsilon^{-2s}
\int_{\mathbb{R}^3}V(x)|q_{\varepsilon,\delta}|^2dx +
\frac{t^4}{4}\varepsilon^{-2s}
\int_{\mathbb{R}^3}|\psi_{|q_{\varepsilon,\delta}|}^t||q_{\varepsilon,\delta}|^2dx - t^rk_\ast\varepsilon^{-2s}\int_{\mathbb{R}^3} |q_{\varepsilon,\delta}|^rdx\\
&\leq&
\varepsilon^{\tau}\left[\frac{C_0}{2}t^{2\sigma}\left(\iint_{\mathbb{R}^{6}}\frac{|q_{\delta}(x)-e^{i(x-y)\cdot
A(\frac{\varepsilon x+\varepsilon y}{2})}q_{\delta}(y)|^2}{|x-y|^{3+2s}}dx dy\right)^{2\sigma} \right.\\
&& \mbox{} \left. +
\frac{t^2}{2}\int_{\mathbb{R}^3}V\left(\varepsilon
x\right)|q_{\delta}|^2dx + \frac{t^4}{4}\varepsilon^{2t}
\int_{\mathbb{R}^3}|\psi_{|q_{\delta}|}^t||q_{\delta}|^2dx -
t^rk_\ast\int_{\mathbb{R}^3} |q_{\delta}|^rdx\right]\\
&\leq& \varepsilon^{\tau}I_\varepsilon(t_0q_\delta).
\end{eqnarray*}
Let $\psi_\delta(x) = e^{iA(0)x}\phi_\delta(x)$, where
$\phi_\zeta(x)$ is as defined above. By  \cite[Lemma 3.6]{bin},
we have the following lemma.
\begin{lemma}\label{lem4.2}   For any $\delta > 0$ there exists $\varepsilon_0 = \varepsilon_0(\delta) >
0$ such that
\begin{eqnarray*}
\iint_{\mathbb{R}^{6}}\frac{|q_{\delta}(x)-e^{i(x-y)\cdot
A(\frac{\varepsilon x+\varepsilon
y}{2})}q_{\delta}(y)|^2}{|x-y|^{3+2s}}dx dy \leq
C\delta^{\frac{6-(3-2s)r}{r}} + \frac{1}{1-s}\delta^{2s} +
\frac{4}{s}\delta^{2s}
\end{eqnarray*}
for all $0 < \varepsilon < \varepsilon_0$ and some constant $C > 0$
depending only on $[\phi]_{s,0}$.
\end{lemma}

On the one hand, since $V(0) = 0$ and  supp\,$\phi_\delta
\subset B_{r_\delta}(0)$, there is $\varepsilon^\ast > 0$ such that
\begin{displaymath}
V\left(\varepsilon x\right) \leq
\frac{\delta}{|\phi_\delta|_2^2}\quad \mbox{for\ all }\ |x| \leq
r_\delta\ \mbox{and}\ 0 < \varepsilon < \varepsilon^\ast.
\end{displaymath}
This fact together with Lemma \ref{lem4.2} implies that
\begin{equation}\label{e4.4}
\max_{t\geq 0} J_\varepsilon(tq_{\delta})  \leq \mathcal
{N}(\delta),
\end{equation}
where
\begin{eqnarray}\label{e4.5}
\mathcal {N}(\delta) &:=&\nonumber
\frac{C_0}{2}t_0^{2\sigma}\left(C\delta^{\frac{6-(3-2s)r}{r}} +
\frac{1}{1-s}\delta^{2s} + \frac{4}{s}\delta^{2s}\right)^{2\sigma}
+ \frac{t_0^2}{2}\delta\\
&& + C\left(C\delta^{\frac{6-(3-2s)r}{r}} + \frac{1}{1-s}\delta^{2s}
+ \frac{4}{s}\delta^{2s} + \delta\right)^{2}.
\end{eqnarray}
Thus we have the following result.
\begin{lemma}\label{lem4.3} Let conditions $(\mathcal {V})$  and $(\mathfrak{M})$ hold. Then for any $\kappa > 0$ there exists $\mathcal {E}_\kappa > 0$
such that for each $0 < \varepsilon < \mathcal {E}_\kappa$, there is
$\widehat{e}_\varepsilon \in E$ with $\|\widehat{e}_\varepsilon\| >
\varrho_\varepsilon$, $J_\varepsilon(\widehat{e}_\varepsilon) \leq
0$ and
\begin{equation}\label{e4.6}
\max_{t\in [0, 1]} J_\varepsilon(t\widehat{e}_\varepsilon) \leq
\kappa\varepsilon^{\tau}.
\end{equation}
\end{lemma}
\begin{proof}
Let $\delta > 0$ satisfy $\mathcal {N}(\delta) \leq \kappa$. Set
$\mathcal {E}_\kappa = \min\{\varepsilon_0,\varepsilon^\ast\}$ and let
$\widehat{t}_\varepsilon > 0$ be such that
$\widehat{t}_\varepsilon\|q_{\varepsilon,\delta}\|_\varepsilon
> \varrho_\varepsilon$ and $J_\varepsilon(tq_{\varepsilon,\delta}) \leq 0$ for
all $t \geq \widehat{t}_\varepsilon$. Choose
$\widehat{e}_\varepsilon = \widehat{t}_\varepsilon
q_{\varepsilon,\delta}.$
Then by \eqref{e4.4}, we can see that the
conclusion of Lemma \ref{lem4.3} holds.
\end{proof}

In order to get the multiplicity results of problem \eqref{e1.1},
one can choose $m^{\ast} \in \mathbb{N}$ functions $\phi_\delta^i
\in C_0^\infty(\mathbb{R}^3)$ such that supp\,$\phi_\delta^i$ $
\cap$ supp\,$\phi_\delta^k = \emptyset$, $i \neq k$,
$|\phi_\delta^i|_s = 1$ and
\begin{displaymath}
\iint_{\mathbb{R}^{6}}\frac{|\phi_\delta^i(x)-\phi_\delta^i(y)|^2}{|x-y|^{3+2s}}dxdy
\leq C\delta^{\frac{6-(3-2s)r}{r}}.
\end{displaymath}
Let $r_\delta^{m^{\ast}}
> 0$ be such that supp\,$\phi_\delta^{i} \subset B_{r_\zeta}^{i}(0)$
for $i = 1,2,\cdots,m^{\ast}$. Set
\begin{equation}\label{e4.7}
q_\delta^i(x) = e^{iA(0)x}\phi_\delta^i(x)
\end{equation}
and
\begin{equation}\label{e4.8}
q_{\varepsilon,\delta}^i(x) =
q_\delta^i(\varepsilon^{-\frac{\tau+2s}{3}}x).
\end{equation}
Denote
\begin{displaymath}
F_{\varepsilon\delta}^{m^{\ast}} =
\mbox{span}\{q_{\varepsilon,\delta}^1, q_{\varepsilon,\delta}^2,
\cdots, q_{\varepsilon,\delta}^{m^{\ast}}\}.
\end{displaymath}
Let $u = \displaystyle\sum_{i=1}^{m^{\ast}}c_i
q_{\varepsilon,\delta}^i \in F_{\varepsilon\delta}^{m^{\ast}}.$
There exists constant $C > 0$ such that
\begin{displaymath}
J_\varepsilon(u) \leq  C\sum_{i=1}^{m^{\ast}}J_{\varepsilon}(c_i
q_{\varepsilon,\delta}^i).
\end{displaymath}
As discussed above,  we have
\begin{displaymath}
J_\varepsilon(c_i q_{\varepsilon,\delta}^i) \leq
\varepsilon^{\tau}I_\varepsilon(|c_i|q_{\delta}^i).
\end{displaymath}
As before, we can obtain the following estimate:
\begin{equation}\label{e4.9}
\max_{u\in F_{\varepsilon\delta}^{m^{\ast}}} J_\varepsilon(u) \leq C
m^\ast\mathcal {N}(\delta)\varepsilon^{\tau}
\end{equation}
for all $\delta$ small enough and some constant $C > 0$.

From \eqref{e4.9}, we have the following lemma.
\begin{lemma}\label{lem4.4}  Let conditions $(\mathcal {V})$  and $(\mathfrak{M})$ hold. Then for any $m^{\ast} \in \mathbb{N}$ and $\kappa > 0$ there
exists $\mathcal {E}_{m^{\ast}\kappa} > 0$ such that for each $0 <
\varepsilon < \mathcal {E}_{m^{\ast}\kappa}$, there exists an
$m^{\ast}$-dimensional subspace $F_{\varepsilon m^{\ast}}$
satisfying
\begin{displaymath}
\max_{u\in F_{\varepsilon\delta}^{m^{\ast}}} J_\varepsilon(u) \leq
\kappa\varepsilon^{\tau}.
\end{displaymath}
\end{lemma}

Now, we began to prove our main results.

\noindent{\bf Proof of Theorem \ref{the1.1}.} For any $0 < \kappa <
\alpha_0$, we choose $\mathcal {E}_\kappa > 0$ so that it satisfies $0 <
\varepsilon < \mathcal {E}_\kappa$, and define the minimax value
$$c_\varepsilon := \inf_{\gamma \in \Upsilon_\varepsilon}\max_{t\in [0,1]} J_\varepsilon(t\widehat{e}_\varepsilon),$$
where $$\Upsilon_\varepsilon := \{\gamma \in C([0, 1], E): \gamma(0)
= 0 \ \mbox{and}\ \gamma(1) = \widehat{e}_\varepsilon\}.$$ By Lemma
\ref{lem4.1}, we have $\alpha_\varepsilon \leq c_\varepsilon \leq
\kappa\varepsilon^{\tau}$. Lemma \ref{lem3.2} implies that
$J_\varepsilon$ satisfies the $(PS)_{c_\varepsilon}$ condition.
Thus, using the mountain pass theorem, there is $u_\varepsilon \in
E$ such that $J'_\varepsilon(u_\varepsilon) = 0$ and
$J_\varepsilon(u_\varepsilon) = c_\varepsilon$, that is
$u_\varepsilon$ is a nontrivial solution of problem  \eqref{e3.1}.

On the other hand,  by \eqref{e3.2}, we have
\begin{eqnarray}\label{e4.11}
\kappa\varepsilon^{\tau}  &\geq&\nonumber J_\varepsilon
(u_\varepsilon) = J_\varepsilon(u_\varepsilon) -
\frac{1}{r}J_\varepsilon'(u_\varepsilon)u_\varepsilon
\\
&\geq&\nonumber \left(\frac{1}{2\sigma}-\frac{1}{r}\right)m_1
[u_\varepsilon]_{s,A}^{2\sigma} +
\left(\frac{1}{2}-\frac{1}{r}\right)\varepsilon^{-2s}
\int_{\mathbb{R}^3}V(x)|u_\varepsilon|^2dx.
\end{eqnarray}
This fact implies that $u_\varepsilon\rightarrow 0$ in $E$ as
$\varepsilon\rightarrow 0$. This completes the proof of Theorem
\ref{the1.1}.\qed

\vspace{2mm}

\noindent{\bf Proof of Theorem \ref{the1.2}.} Denote the set of all
symmetric (in the sense that $-Z = Z$) and closed subsets of $E$ by
$\Sigma$, for each $Z \in \Sigma$. Denote by gen$(Z)$ the
Krasnoselski genus, and define
\begin{displaymath}
j(Z) := \min_{\iota\in
\Gamma_{m^\ast}}\mbox{gen}(\iota(Z)\cap\partial
B_{\varrho_\varepsilon}),
\end{displaymath}
where $\Gamma_{m^\ast}$ is the set of all odd homeomorphisms $\iota
\in C(E, E)$ and $\varrho_\varepsilon$ is given by Lemma
\ref{lem4.1}. Then $j$ is a version of Benci's pseudoindex
\cite{benci}. Let
\begin{displaymath}
c_{\varepsilon i} := \inf_{j(Z)\geq i}\sup_{u\in Z}J_\varepsilon(u),
\quad 1 \leq i \leq m^\ast.
\end{displaymath}
Since $J_\varepsilon(u) \geq \alpha_\varepsilon$ for all $u \in
\partial B_{\varrho_\varepsilon}^{+}$ and since $j(F_{\varepsilon m^\ast}) =
\dim F_{\varepsilon m^\ast} = m^\ast$, we obtain
\begin{displaymath}
\alpha_\varepsilon \leq c_{\varepsilon 1} \leq \cdots\leq
c_{\varepsilon m^\ast} \leq \sup_{u \in H_{\varepsilon m^\ast}}
J_\varepsilon(u) \leq \kappa\varepsilon^{\tau}.
\end{displaymath}
Thus, Lemma \ref{lem3.2} implies  that $J_\varepsilon$ satisfies the
$(PS)_{c_\varepsilon}$ condition for any $c <
\alpha_0\varepsilon^{\tau}$. By using the mountain pass theorem, we
know that $c_{\varepsilon i}$ are critical valves and
$J_\varepsilon$ has at least $m^\ast$ pairs of nontrivial critical
points satisfying
\begin{displaymath}
\alpha_\varepsilon \leq J_\varepsilon(u_\varepsilon) \leq
\kappa\varepsilon^{\tau}.
\end{displaymath}
Hence, problem \eqref{e1.1} has at least $m^\ast$ pairs of
solutions. Moreover, we also have $u_{\varepsilon,i}\rightarrow 0$
in $E$ as $\varepsilon \rightarrow 0$, $i=1,2,\cdots,m$.
$\hfill\Box$

%%%%%%%%%%%%%%%%%%%%%%%%%%%%%%%%%%%%%%%%%%%%%%%%%%%%%%%%%%%%%%%%%%%%%%%%%
%%%%%%%%%%%%%%%%%%%%%%%%%%%%%%%%%%%%%%%%%%%%%%%%%%%%%%%%%%%%%%%%%%%%%%%%%
\section*{Funding}
Z. Zhang was supported by 2020 Nanjing Vocational College of
Information Technology doctoral special fund project ``Study on the
Climate Ensemble Forecast Model Based on the East Asian Monsoon
Region'' (YB20200902).  D.D. Repov\v{s} was supported by the
Slovenian Research Agency grants (Nos. P1-0292, N1-0114, N1-0083).

%%%%%%%%%%%%%%%%%%%%%%%%%%%%%%%%%%%%%%%%%%%%%%%%%%%%%%%%%%%%%%%%%%%%%%%%%
%%%%%%%%%%%%%%%%%%%%%%%%%%%%%%%%%%%%%%%%%%%%%%%%%%%%%%%%%%%%%%%%%%%%%%%%%

\end{document}